\documentclass[a4paper,12pt,reqno]{amsart}

\usepackage{amsfonts}
\usepackage{amsmath}
\usepackage{amssymb}
\usepackage{lipsum}
\usepackage{mathrsfs}
\usepackage{hyperref}
\usepackage{enumerate}
\numberwithin{equation}{section}

\usepackage{graphicx}
\usepackage{enumitem}

\newtheorem{theorem}{Theorem}[section]
\newtheorem{defi}[theorem]{Definition}
\newtheorem{remark}[theorem]{Remark}
\newtheorem{Cor}[theorem]{Corollary}

\newtheorem{lemma}[theorem]{Lemma}

\newtheorem{rem}[theorem]{Remark}

\def\C{{\mathbb C}}

\def\R2n{{\mathbb R}^{2n}}

\def\C{{\mathbb C}}

\def\R2{{\mathbb R}^2}
\def\R2n{{\mathbb R}^{2n}}

\def\HS{{\mathtt{HS}}}
\def\N{{\mathbb N}}

\usepackage{fancyhdr}
\begin{document}
\title[Eigenfunction expansions of ultradifferentiable functions and ultradistributions. III. Hilbert spaces and Universality]
{Eigenfunction expansions of ultradifferentiable functions and ultradistributions. III. Hilbert spaces and Universality}

\author[Aparajita Dasgupta]{Aparajita Dasgupta}
\address{
  Aparajita Dasgupta:
   \endgraf
  Department of Mathematics
  \endgraf
  Indian Institute of Technology, Delhi, Hauz Khas
  \endgraf
  New Delhi-110016 
  \endgraf
  India
  \endgraf
  {\it E-mail address} {\rm adasgupta@maths.iitd.ac.in}
  }

\author[Michael Ruzhansky]{Michael Ruzhansky}
\address{
  Michael Ruzhansky:
  \endgraf
  Department of Mathematics: Analysis, Logic and Discrete Mathematics
  \endgraf
  Ghent University, Belgium
  \endgraf
  and
  \endgraf
  School of Mathematical Sciences
    \endgraf
    Queen Mary University of London
  \endgraf
  United Kingdom
  \endgraf
  {\it E-mail address} {\rm ruzhansky@gmail.com}
  }

\thanks{The second
 author was partly supported by the FWO Odysseus Project, by the EPSRC Grant EP/K039407/1, and by 
 the Leverhulme Research Grant RPG-2014-02. 
 }
\date{\today}


\begin{abstract}
In this paper we analyse the structure of the spaces of smooth type functions, generated by elements of arbitrary Hilbert spaces, as a continuation of the research in our papers \cite{DaR2} and \cite{DaR3}. We prove that these spaces are perfect sequence spaces. As a consequence we describe the tensor structure of sequential mappings on the spaces of smooth type functions and characterise their adjoint mappings. As an application we prove the universality of the spaces of smooth type functions on compact manifolds without boundary.
\end{abstract}

\maketitle

\tableofcontents

\section{Introduction}
The present paper is a continuation of our papers \cite{DaR2} and \cite{DaR3}.  In \cite{DaR3}, we analysed the structure of the spaces of coefficients of eigenfunction expansions of functions in Komatsu classes on compact manifolds. We also described the tensor structure of sequential mappings on spaces of Fourier coefficients and characterised their adjoint mappings. In particular, these classes include spaces of analytic and Gevrey functions, as well as spaces of ultradistributions, dual spaces of distributions and
ultradistributions, in both Roumieu and Beurling settings.  In another work, \cite{DaR2}, we have characterised Komatsu spaces of ultradifferentiable functions and ultradistributions on compact
manifolds in terms of the eigenfunction expansions related to positive elliptic operators. Here we note that, using properties of the elliptic operators and the Plancherel formula one can get such type of characterisation of smooth functions in terms of their Fourier coefficients. For example,  if $E$ is a positive elliptic pseudo-differential operator on a compact manifold $X$ without boundary and $\lambda_j$ denotes its eigenvalues in the ascending order, then smooth functions on $X$ can be characterised in terms of their Fourier
coefficients:
\begin{equation}
f\in C^{\infty}(X) \Leftrightarrow \forall N~ \exists C_N : |\widehat{f}(j,l)|\leq C_N \lambda_j^{-N}~\text{for~ all~} j\geq 1, 1\leq l\leq d_j,
\end{equation}
where $\widehat{f}(j,l)=\left(f,e^{l}_j\right)_{L^2}$ with $e^{l}_j$ being the $l^{th}$ eigenfunction corresponding to the eigenvalue $\lambda_j$ (of multiplicity $d_j$). Such characterisations for analytic functions were obtained by Seeley in \cite{see:exp}, with a subsequent extension to Gevrey and, more generally, to Komatsu classes, in \cite{DaR2}.
The  results  obtained in \cite{DaR3} do not include the cases of smooth functions on compact manifolds. We will extend the results in \cite{DaR3} to the spaces of smooth functions. Moreover in this work, we aim at  discussing an abstract analysis of the spaces of smooth type functions generated by basis elements of an arbitrary Hilbert space. Considering an abstract point of view has an advantage that the results will cover the analysis of smooth functions on different spaces like compact Lie groups and manifolds. 
In particular, we introduce a notion of smooth functions generated by elements of a Hilbert space $\mathcal{H}$ forming a basis. We will show that the appearing spaces of
coefficients with respect to expansions in eigenfunctions of positive self-adjoint operators
are perfect spaces in the sense of the theory of sequence spaces (see, e.g., K\"{o}the \cite{Kothe:BK-top-vector-spaces-I}).
Consequently, we obtain tensor representations for linear mappings between spaces
of smooth type functions.
Such discrete representations in a given basis are useful in different areas of time-frequency
analysis, in partial differential equations, and in numerical investigations. 

Using the obtained representations we establish the universality properties of the appearing spaces. In \cite{Wael}, L. Waelbroeck proved the so-called {\em universality} of the space of Schwartz distributions $\mathcal{E}(V)$ with compact support on a $C^{\infty}$-manifold $V,$ with the $\delta$-mapping $\delta: V\rightarrow \widehat{\mathcal{E}(V)},$ that is, any vector valued $C^{\infty}$-mapping $f: V\rightarrow E,$ from $V$ to a sequence space $E$, factors through $\delta: V\rightarrow \widehat{\mathcal{E}(V)}$ by a unique linear morphism $\tilde{f}:\widehat{\mathcal{E}(V)}\rightarrow E$ as $f=\tilde{f}\circ\delta.$ The universality of the spaces of Gevrey functions on the torus has been established in \cite{Tag2}. As an application of our tensor representations, we prove the universality of the spaces of smooth functions on compact manifolds. 

Our analysis is based on the global Fourier analysis on arbitrary Hilbert spaces using techniques similar to compact manifold which was consistently developed in \cite{DR}, with a number of subsequent applications, for example to the spectral properties of operators \cite{Delgado-Ruzhansky:JFA-2014}, or to the wave equations for the Landau Hamiltonian \cite{RT:LMP}. The corresponding version of the Fourier analysis is
based on expansions with respect to orthogonal systems of eigenfunctions of a self-adjoint
operator. The non self-adjoint version has been developed in \cite{RT:IMRN}, with a subsequent extension in \cite{RT:MMNP}.

The paper is organised as follows. In Section \ref{FAonHS} we will briefly recall the constructions leading to the global Fourier analysis on arbitrary Hilbert spaces and define the smooth type function spaces.In Section \ref{SEC:seqspaces} we very briefly recall the relevant definitions from the theory of sequence spaces. In Section \ref{TR1} we present the main results of this paper and their proofs. In Section \ref{Universality} we prove the universality results for the smooth  functions on compact manifold.

 \section{Fourier Analysis on Hilbert Spaces}\label{FAonHS}
  
Let $(\mathcal{H}, ||\cdot||_{\mathcal{H}})$ be a separable Hilbert space and denote by 
$$\mathcal{U}:=\{e_{jl}: e_{jl}\in \mathcal{H}, 1\leq l\leq d_j, d_j\in\mathbb{N}\}_{j\in \mathbb{N}}$$  a collection of elements of $\mathcal{H}$ . 
 We assume that $\mathcal{U}$ is a basis of the space $\mathcal{H}$  with the property $$(e_{jl},e_{mn})_{\mathcal{H}}=\delta_{jm}\delta_{ln}, ~~j,m\in\mathbb{N} ~\text{and}~~1\leq l\leq d_j,~1\leq n\leq d_m,$$ where $\delta_{jm}$ is the Kroneckar delta, equal to 1 for $j=m$, and to zero otherwise.
Also let us fix a sequence of positive numbers  $\Lambda:=\{\lambda_{j}\}_{j\in\mathbb{N}}$  such that  $0<\lambda_{1}\leq \lambda_2 \leq \lambda_3\leq ...,$ and the series \begin{equation}\label{dimineq}\sum\limits_{j={1}}^{\infty}d_j\lambda^{-s_0}_j<\infty \end{equation} converges for some $s_0>0.$  For example, in a compact $C^{\infty}$ manifold $X$ of dimension $n$ without boundary and with a fixed measure we have
$$\sum\limits_{j=1}^{\infty}d_j(1+\lambda_j)^{-q}<\infty ~\textrm{if~and~only~if}~ q>\frac{n}{\nu},$$ where
$0<\lambda_1<\lambda_2<...$ are eigenvalues of a positive elliptic pseudo-differential operator $E$ of an integer order $\nu$, with $H_j\subset L^{2}(X)$ the corresponding  eigenspace and  $$d_j := \dim H_j, ~H_0:= \textrm{ker} E,~\lambda_0:= 0,~d_0:=\dim H_0.$$  
We associate to the pair $\{\mathcal{U},\Lambda\}$ a linear self-adjoint operator $E: \mathcal{H}\rightarrow\mathcal{H}$ such that 
\begin{equation}\label{ellidef}E^{s}f=\sum\limits_{j=1}^{\infty}\sum\limits_{l=1}^{d_j}\lambda_{j}^{s}(f,e_{jl})e_{jl}\end{equation}  for $s\in\mathbb{R}$ and those $f\in\mathcal{H}$ for which the series converges in $\mathcal H.$ Then $E$ is densely defined since
$$E^{s}e_{jl}=\lambda_{j}^{s}e_{jl}, ~~1\leq l\leq d_j, ~j\in\mathbb{N},$$ and
 $\mathcal{U}$ is a basis of $\mathcal{H}$. Also we  write $H_{j}=\text{span}\{e_{jl}\}_{1\leq l\leq d_j},$ and so $\dim H_j= d_j.$  Then we have  $$\mathcal{H}=\bigoplus\limits_{j\in\mathbb{N}}H_{j}.$$ \\
 The Fourier transform for $f\in\mathcal{H}$ is defined as
 $$\widehat{f}(j,l):=(f,e_{jl})_{\mathcal{H}},~~j\in\mathbb{N}, ~1\leq l\leq d_j.$$
 We next define the following notions:\\~\\
 \vspace{.1cm}
 The spaces of smooth type functions are defined by
 $$H^{\infty}_{E}:= \bigcap\limits_{s\in\mathbb{R}}H^{s}_{E},$$
 where
 $$H^{s}_E:=\left\{\phi\in\mathcal{H} : \sum\limits_{j=1}^{\infty}\lambda^{2s}_{j}||\widehat{\phi}(j)||^{2}_{{\HS}}<\infty \right\}.$$ 
 There exists a linear pairing $$(f,g)_{H^{s}_E} = ( E^{s}f, E^{-s}g)_{\mathcal{H}},$$ for $f\in H^{s}_E$ and $g\in H^{-s}_E.$ It is easy to see from this that every continuous linear functional on $H^{s}_{E}$ is of the form $f\rightarrow( f,g)_{H^{s}_E}$ for some $g\in H^{-s}_{E},$ that is  $(H^{s}_{E})^{\prime}=H^{-s}_{E}.$  Then we denote the space of distributions as $\left(H^{\infty}_{E}\right)^{\prime}=\bigcup\limits_{s\in\mathbb{R}}H^{-s}_{E}$.\\
 
  \section{Sequence spaces and sequential linear mappings}
  \label{SEC:seqspaces}
  
  We briefly recall that a sequence space $V$ is a linear subspace of 
$$\mathbb{C}^{\mathbb Z}=\{a=(a_j)|a_j\in\mathbb{C}, j\in \mathbb{Z}\}.$$
The dual $\widehat{V}$ ($\alpha$-dual in the terminology of G. K\"{o}the \cite{Kothe:BK-top-vector-spaces-I}) is a sequence  space defined by
$$\widehat{V}=\{a\in \mathbb{C}^{\mathbb Z}: \sum_{j\in \mathbb{Z}} |f_j||a_j|<\infty
\textrm{ for all }f\in V\}.$$

A sequence space $V$ is called {\em perfect} if $\widehat{\widehat{V}}=V$.
A sequence space is called {\em normal} if $f=(f_j)_{j\in\mathbb{N}}\in V$ implies $|f|=(|f_j|)_{j\in\mathbb{N}}\in V.$
A dual space $\widehat{V}$ is normal so that any perfect space is normal.

A pairing ${\langle\cdot,\cdot\rangle}_{V}$ on $V$ is a bilinear function on $V\times\widehat{V}$ defined by 
$$\langle f,g\rangle_{V}=\sum_{j\in \mathbb{Z}}{f_j g_j}\in\mathbb{C},$$ 
which converges absolutely by the definition of $\widehat{V}.$

\begin{defi}\label{seqdefi} $\phi: V\rightarrow \mathbb{C}$ is called a {\em sequential linear functional} if there exists some $a\in\widehat{V}$ such that $\phi(f)=\langle f,a\rangle_V$ for all $f\in V.$ We abuse the notation by also writing $a: V\rightarrow \mathbb{C}$ for this mapping.\end{defi}

\begin{defi} A mapping $\phi:V\rightarrow W$ between two sequence spaces is called a {\em sequential linear mapping} if 
\begin{enumerate}
\item $\phi$ is algebraically linear,
\item for any $g\in \widehat W,$ the composed mapping $g\circ \phi: V\rightarrow \mathbb{C} $ is in $\widehat{V}.$ 
\end{enumerate}
\end{defi}
 
 \section{Tensor representations and the adjointness}\label{TR1}
 
 In this section we discuss $\alpha$-duals of the spaces, tensor representations for mappings between these spaces and their $\alpha$-duals, and obtain the corresponding 
 adjointness theorem.
 
 \subsection{Duals and $\alpha$-duals}\hfill\\~~~
 
  In this section we  first prove that the  $\alpha$-dual, $\widehat{H^{s}_{E}}$ of the space $H^{s}_{E},$ where
  $$\widehat{H^{s}_{E}}=\left\{v=(v_{j})_{j\in\mathbb{N}}, v_j\in \mathbb{C}^{d_j}; \sum\limits_{j=1}^{\infty}\sum\limits_{l=1}^{d_j}|\widehat{\phi}(j,l)||v_{jl}|<\infty, ~\text{for ~all~}\phi\in H^{s}_{E} \right\},$$ 
coincides with the space $H^{-s}_E.$
 
 \begin{remark}\label{remest1}
 Here we observe that,
 $$w\in\widehat{H^{s}_{E}}\implies \sum\limits_{j=1}^{\infty}\lambda_{j}^{-s-s_0/2}||w_{j}||_{\mathtt{HS}}<\infty.$$\\
  Indeed, let $w\in\widehat{H^{s}_{E}}.$ Define $\widehat{\phi}(j,l)= \lambda_{j}^{-(s+s_0/2)}$ for $j\in\mathbb{N}$ and $1\leq l\leq d_j.$\\ Then $||\widehat{\phi}(j)||_{\mathtt{HS}}=d_{j}^{1/2}\lambda_{j}^{-s-s_0/2}.$ It follows from \eqref{dimineq} and the definition of $H^{s}_E$ that $\phi\in H^{s}_{E}.$\\
  From the definition of the $\alpha$-dual and using the inequality, $$||w_j||_{\mathtt{HS}}\leq ||w_j||_{l^{1}},$$  we can then conclude that,
  \begin{eqnarray}
 \sum\limits_{j=1}^{\infty}\lambda_{j}^{-s-s_0/2}||w_{j}||_{\mathtt{HS}}&\leq& \sum\limits_{j=1}^{\infty}\lambda^{-s-s_0/2}_{j}||w_{j}||_{l^1}\nonumber\\
&=& \sum\limits_{j=1}^{\infty}\sum_{l=1}^{d_j}|\widehat{\phi}(j,l)||w_{jl}|<\infty.
  \end{eqnarray}
  \end{remark}
  Our first result is the identifiction of the topological dual with the $\alpha$-dual.
 \begin{theorem}\vspace{.3cm}
 $\widehat{H^{s}_{E}}=H^{-s}_E.$
 \end{theorem}
  \begin{proof}
   First we will show $H^{-s}_{E}\subseteq \widehat{H^{s}_{E}}. $\\
   Let $u\in H^{-s}_{E}.$\\ Then from the definition we have
   $\sum\limits_{j=1}^{\infty}\lambda_{j}^{-2s}||\widehat{u}(j)||^{2}_{\HS}<\infty.$ 
   We denote for $j\in\mathbb{N},$ $u_{jl}=\widehat{u}(j,l)\in \mathbb{C}^{d_j},$ $l=1,2,...,d_j.$ Using the Cauchy-Schwartz inequality, for any $\phi\in H^{s}_E$ we get
 
  \begin{multline*}
 \sum\limits_{j=1}^{\infty}\sum\limits_{l=1}^{d_j}|\widehat{\phi}(j,l)||u_{jl}| \\
 \leq \left(\sum\limits_{j=1}^{\infty}\sum\limits_{l=1}^{d_j}\lambda^{2s}_{j}|\widehat{\phi}(j,l)|^2\right)^{1/2}\left(\sum\limits_{j=1}^{\infty}\sum\limits_{l=1}^{d_j}\lambda^{-2s}_{j}|u_{jl}|^2\right)^{1/2}<\infty,~\text{(by~ assumption)}.
\end{multline*}
  This implies  that $u \in \widehat H^{s}_E.$  We thus obtain $H^{-s}_{E}\subseteq \widehat{H^{s}_{E}}. $ \\
   Next we will show that $\widehat{H^{s}_{E}}\subseteq H^{-s}_E.$\\
   By duality we know that $(H^{s}_{E})^{\prime}=H^{-s}_{E}.$ So it will be enough if we can prove $\widehat{H^{s}_{E}}\subseteq (H^{s}_{E})^{\prime}.$\\
   
   Let $w\in \widehat{H^{s}_{E}}.$ Then for $\phi \in {H^{s}_{E}}$ we define $$w(\phi)=\sum\limits_{j=1}^{\infty}\lambda_{j}^{-s_0/2}\widehat{\phi}(j)\cdot w_{j}=\sum\limits_{j=1}^{\infty}\lambda_{j}^{-s_0/2}\sum_{l=1}^{d_j}\widehat{\phi}(j,l)w_{jl}.$$
   Then using Remark \ref{remest1} we have
   \begin{eqnarray}|w(\phi)|&\leq & \sum\limits_{j=1}^{\infty}\lambda_{j}^{-s_0/2}||\widehat{\phi}(j)||_{\mathtt{HS}}||w_{j}||_{\mathtt{HS}}\nonumber\\
   &\leq& C\sum\limits_{j=1}^{\infty}\lambda_{j}^{-s_0/2}\lambda_{j}^{-s}||w_{j}||_{\mathtt{HS}}<\infty,\nonumber\\
   \end{eqnarray} since $\phi\in H^{s}_{E}$.  So $w(\phi)$ is well defined.\\
 
  We next check that $w$ is continuous.  Let $\phi_m\rightarrow \phi$ as $m\rightarrow \infty$ in $H^{s}_E.$ This means,
   $||\phi_m-\phi||_{H^{s}_E}\rightarrow 0$ as $m\rightarrow\infty,$ which implies $\lambda_{j}^{s}||\widehat{\phi_m}(j)-\widehat{\phi}(j)||_{\mathtt{HS}}\rightarrow 0$ as $m\rightarrow \infty,$ for $j\in\mathbb{N}$. So we have
   $$||\widehat{\phi_m}(j)-\widehat{\phi}(j)||_{\mathtt{HS}}\leq C_m\lambda^{-s}_{j},$$ where $C_m\rightarrow 0$ as $m\rightarrow\infty.$ Then
   \begin{eqnarray}
   |w(\phi_m-\phi)|&\leq& \sum\limits_{j=1}^{\infty}\lambda_{j}^{-s_0/2}||\widehat{\phi_m}(j)-\widehat{\phi}(j)||_{\mathtt{HS}}||w_j||_{\mathtt{HS}}\nonumber\\
   &\leq& C_{m} \sum\limits_{j=1}^{\infty}\lambda_{j}^{-(s+s_0/2)}||w_j||_{\mathtt{HS}}\rightarrow 0,\nonumber
   \end{eqnarray}
   as $m\rightarrow \infty.$ Hence $w$ is continuous. This gives $w\in (H^{s}_E)^{\prime}=H^{-s}_E.$ So we have
   $\widehat{H^{s}_{E}}\subseteq H^{-s}_{E},$ that implies $\widehat{H^{s}_{E}}= H^{-s}_{E}.$ 
   \end{proof}
 From this we can have the following corollary.
 \begin{Cor}
 $v\in \widehat{H^{s}_E}$ $\iff$ $\sum\limits_{j=1}^{\infty}\lambda_{j}^{-2s}||v_j||^{2}_{\mathtt{HS}}<\infty$  and $v_{j}\in\mathbb{C}^{d_j},$ we denote  $v_{jl}=\widehat{v_{jl}}=(v,e_{jl})_{\mathcal{H}}.$
 \end{Cor}
 \begin{remark}\label{Realphchar}  From the definition of $H^{\infty}_E$, 
 $$v\in H^{\infty}_E \iff \sum\limits_{j=1}^{\infty}\lambda_{j}^{2s}||v_j||_{\mathtt{HS}}^2<\infty, ~~\forall {s\in\mathbb{R}},$$ where $v_{j}=\widehat{v}(j)\in \mathbb{C}^{d_j}.$ \end{remark}
 We next define the $\alpha$-dual of the space $\widehat{H^{\infty}_E},$
 $$\widehat{H^{\infty}_E}=\left\{v=(v_j), v_j\in\mathbb{C}^{d_j}: \sum\limits_{j=1}^{\infty}\sum\limits_{l=1}^{d_j}|v_{jl}||\widehat{\phi}(j,l)|<\infty, \textrm{~for~all}~\phi\in H^{\infty}_{E}\right\}.$$
 
  Also observe that  $\widehat{H^{\infty}_E}=\widehat{\bigcap\limits_{s\in\mathbb{R}}H^{s}_{E}}=\bigcup\limits_{s\in\mathbb{R}}\widehat{H^{s}_E}=\bigcup\limits_{s\in\mathbb{R}}H^{-s}_{E}.$ From this we can state the following lemma and the proof will follow from our above observation.
 \begin{lemma} $v\in \widehat{H^{\infty}_E}\iff$ for some $s\in\mathbb{R}$ we have $\sum\limits_{j=1}^{\infty}\lambda_{j}^{-2s}||v_j||^2_{\mathtt{HS}}<\infty$.\end{lemma}
 Next we proceed to prove that $H^{\infty}_E$ is a perfect space. But before that let us prove the following lemma.
 \begin{lemma}\label{alphachar}
We have $w\in[\widehat{H^{\infty}_{E}}]^{\wedge}$ if and only if $\sum\limits_{j=1}^{\infty}\lambda_{j}^{2s}||w_j||^{2}_{\mathtt{HS}}<\infty$ for all $s\in\mathbb{R}.$  
 \end{lemma}
 \begin{proof}
 Let $w\in [\widehat{H^{\infty}_{E}}]^{\wedge}.$ Let $s\in\mathbb{R},$ we define
 $$v_{jl}=\widehat{v}(j,l)=\lambda_{j}^{s},~~1\leq l\leq d_j.$$
 Then $||v_{j}||^{2}_{\mathtt{HS}}=d_j\lambda_{j}^{2s}$ and so $v\in H^{-s-\frac{s_0}{2}}_{E}$  because in view of \eqref{dimineq} we have 
 \begin{equation}\sum_{j=1}^{\infty}\lambda_j^{-2s}\lambda_j^{-s_0}||v_{j}||^{2}_{\mathtt{HS}}=\sum_{j=1}^{\infty}d_j\lambda_j^{-2s}\lambda_j^{-s_0}\lambda_{j}^{2s}=\sum_{j=1}^{\infty}d_j\lambda_j^{-s_0}<\infty.\end{equation}
 Then $v\in \widehat{H^{\infty}_E}=\bigcup\limits_{s\in\mathbb{R}}H^{-s}_{E}$ which gives $\sum\limits_{j=1}^{\infty}\sum\limits_{l=1}^{d_j}|w_{jl}||v_{jl}|<\infty.$ 
Now we observe first that
 \begin{eqnarray}\label{reqineq}
 \sum\limits_{j=1}^{\infty}\lambda_{j}^{2s}||w_j||^{2}_{\mathtt{HS}}&\leq& \sum\limits_{j=1}^{\infty}d_j\lambda_{j}^{2s}||w_j||^{2}_{\mathtt{HS}}\nonumber\\
 &=& \sum\limits_{j=1}^{\infty}||v_j||_{\mathtt{HS}}^2||w_j||^{2}_{\mathtt{HS}},
\end{eqnarray}
since $d_j\geq 1.$

We want to show that $\sum\limits_{j=1}^{\infty}||v_j||_{\mathtt{HS}}^2||w_j||^{2}_{\mathtt{HS}}<\infty.$
 To prove this we will use the following identity:
 $$\sum_{i=1}^n a_i^2\sum_{l=1}^n b_{l}^2=\sum_{i=1}^n a_i^2 b_{i}^2+\sum_{i=1}^n a_{i}^2\left(\sum_{l=1}^n b_{l}^2-b_{i}^2\right).$$ From this we get
 $$\sum_{i=1}^{d_j}|w_{ji}|^2\sum_{l=1}^{d_j}|v_{jl}|^2=\sum_{i=1}^{d_j}|v_{ji}|^2|w_{ji}|^2+\sum_{i=1}^{d_j}|w_{ji}|^2\left(\sum_{l=1}^{d_j}|v_{jl}|^2-|v_{ji}|^2\right).$$
 
 We consider the second term of the above inequality, that is,
 $$\sum_{i=1}^{d_j}|w_{ji}|^2\left(\sum_{l=1}^{d_j}|v_{jl}|^2-|v_{ji}|^2\right)\leq \sum_{i=1}^{d_j}|w_{ji}|^2\sum_{l=1}^{d_j}|v_{jl}|^2 \leq\sum_{i=1}^{d_j}|w_{ji}|^2\left(C\lambda_{j}^{2s+s_0}\right),$$ since  $v\in H^{-s-s_0/2}_{E}.$  Then  we get
 \begin{eqnarray}\label{ineqest1}
 \sum_{j=1}^\infty||v_j||_{\mathtt{HS}}^2||w_j||^{2}_{\mathtt{HS}}&\leq& \sum_{j=1}^{\infty}\sum_{i=1}^{d_j} |w_{ji}|^2\left(|{v}_{ji}|^2+C\lambda_{j}^{2s+s_0}\right)\nonumber\\
 &\leq &\sum_{j=1}^{\infty}\sum_{i=1}^{d_j} |w_{ji}|^2|u_{ji}|^2\nonumber\\
 &\leq& C \left(\sum_{j=1}^{\infty}\sum_{i=1}^{d_j} |w_{ji}||u_{ji}|\right)^{2},
 \end{eqnarray}
 where $|u_{ji}|^2=|{v}_{ji}|^2+C\lambda_{j}^{2s+s_0}.$ 
 Now
 \begin{eqnarray}\sum\limits_{j=1}^{\infty}\lambda_j^{-2s-2s_0}||u_j||_{\mathtt{HS}}^2&=&\sum\limits_{j=1}^{\infty}\lambda_j^{-2s-2s_0}||v_j||_{\mathtt{HS}}^2+C\sum\limits_{j=1}^{\infty}d_j\lambda_j^{-2s-2s_0}\lambda_{j}^{2s+s_0}\nonumber\\
 &\leq& C^{\prime}\sum\limits_{j=1}^{\infty}\lambda_j^{-s_0}+ C\sum\limits_{j=1}^{\infty} d_j\lambda_j^{-s_0}<\infty,\end{eqnarray}
 which implies that $u\in H^{-s-s_0}_{E},$ that is $u\in \widehat{H^{\infty}_{E}}.$ This gives that 
 \begin{equation}\label{newes1}
 \sum_{j=1}^{\infty}\sum_{i=1}^{d_j} |w_{ji}||u_{ji}|<\infty,
 \end{equation}
 as $w\in [\widehat{H^{\infty}_{E}}]^{\wedge}$ and $u\in \widehat{H^{\infty}_{E}}. $
 So we have from \eqref{reqineq}, \eqref{ineqest1} and \eqref{newes1} that
 $$ \sum\limits_{j=1}^{\infty}\lambda_{j}^{2s}||w_j||^{2}_{\mathtt{HS}}<\infty.$$
 
 Next we proceed to prove the opposite direction. Let $$w\in \Sigma=\left\{v=(v_{j})_{j\in\mathbb{N}}, ~v_{j}\in\mathbb{C}^{d_j}\right\}$$ be such that $\sum\limits_{j=1}^{\infty}\lambda_{j}^{2s}||w_j||^{2}_{\mathtt{HS}}<\infty$ for all $s\in\mathbb{R}.$ Let $v\in \widehat{H^{\infty}_E}=\bigcup\limits_{s\in\mathbb{R}} H^{-s}_{E}.$ In particular, we have $v\in H^{-s}_{E}$ for some $s\in\mathbb{R}.$ We have to show 
 $$\sum\limits_{j=1}^{\infty}\sum_{l=1}^{d_j}|w_{jl}| |v_{jl}|<\infty.$$ 
 By the Cauchy-Schwartz inequality we have
 \begin{eqnarray}
\sum_{j=1}^{\infty}\sum_{l=1}^{d_j}|w_{jl}||v_{jl}|&=& \sum_{j=1}^{\infty} \sum_{l=1}^{d_j}\lambda^{s}_j|w_{jl}|\lambda^{-s}_j |v_{jl}|\nonumber\\
&\leq& \left(\sum_{j=1}^{\infty} \sum_{l=1}^{d_j}\lambda^{2s}_j|w_{jl}|^2\right)^{1/2}\left(\sum_{j=1}^{\infty} \sum_{l=1}^{d_j}\lambda^{-2s}_j|v_{jl}|^2\right)^{1/2}<\infty.
 \end{eqnarray}
 It follows that $w\in  [\widehat{H^{\infty}_{E}}]^{\wedge},$ completing the proof.\end{proof}~\\~
 Now using Remark \ref{Realphchar} and Lemma \ref{alphachar} we can prove that the spaces $H^{\infty}_{E}$ are perfect spaces.

 \begin{theorem}\label{P:perfect}
 $H^{\infty}_{E}$ is a perfect space.
 \end{theorem}
\begin{proof}
From the definition we always have $H^{\infty}_{E}\subseteq [\widehat{H^{\infty}_{E}}]^{\wedge}.$ We will prove the other direction.
Let  $w\in  [\widehat{H^{\infty}_{E}}]^{\wedge},$ $w=(w_j)_{j\in\mathbb{N}}$ and $w_j\in\mathbb{C}^{d_j}.$\\
Define
$$\phi=\sum_{j=1}^{\infty}\sum_{l=1}^{d_j}w_{jl}e_{jl}.$$ The series is convergent and  $\phi\in\mathcal{H}$ since
\begin{eqnarray}
||\phi||_{\mathcal{H}}&=&\left|\left|\sum_{j=1}^{\infty}\sum_{l=1}^{d_j}w_{jl}e_{jl}\right|\right|_{\mathcal{H}}\nonumber\\
&\leq& \left(\sum_{j=1}^{\infty}\sum_{l=1}^{d_j}\lambda_{j}^{s_0}|w_{jl}|^2\right)^{1/2}
\left(\sum_{j=1}^{\infty}\sum_{l=1}^{d_j}\lambda_{j}^{-s_0}||e_{jl}||_{\mathcal{H}}^2\right)^{1/2}\nonumber\\
&=& \left(\sum_{j=1}^{\infty}\sum_{l=1}^{d_j}\lambda_{j}^{s_0}|w_{jl}|^2\right)^{1/2}\left(\sum_{j=1}^{\infty}d_j\lambda_{j}^{-s_0}\right)^{1/2}
<\infty,
\end{eqnarray}$  $
since $w\in  [\widehat{H^{\infty}_{E}}]^{\wedge}$ and using \eqref{dimineq} and Lemma \ref{alphachar}.

 Also from the property $(e_{jl},e_{mn})_{\mathcal{H}}=\delta_{jm}\delta_{ln},$ for $j,m\in\mathbb{N}$ and $1\leq l\leq d_j,$ $1\leq n\leq d_m,$ it is obvious that $\widehat{\phi}(j,l)=(\phi,e_{jl})_{\mathcal{H}}=w_{jl}.$ This gives $||\widehat{\phi}(j)||_{\mathtt{HS}}=||w_j||_{\mathtt{HS}}.$ So by Lemma \ref{alphachar},
\begin{eqnarray}
w\in [\widehat{H^{\infty}_{E}}]^{\wedge}&\implies&\sum\limits_{j=1}^{\infty}\lambda_{j}^{2s}||w_j||^{2}_{\mathtt{HS}}<\infty, ~\text{for~all~}s\in\mathbb{R}\nonumber\\
&\implies&\sum\limits_{j=1}^{\infty}\lambda_{j}^{2s}||\widehat{\phi}(j)||^{2}_{\mathtt{HS}}<\infty,\nonumber
\end{eqnarray}
and so from Remark \ref{Realphchar} we have $w\in H^{\infty}_E$  which implies that $ [\widehat{H^{\infty}_{E}}]^{\wedge}\subseteq H^{\infty}_E$ holds.
\end{proof}

\subsection{Adjointness}\label{TR2}\hfill\\~~~

   Before proving the adjointness theorem we first prove the following lemma,
\begin{lemma} Let $v\in H^{\infty}_E$ and $w\in \widehat{H^{\infty}_E}$. Then we have
\begin{equation}\label{tensoreq}
\sum_{j=1}^{\infty}||v_j||_{\mathtt{HS}}||w_j||_{\mathtt{HS}}<\infty, \end{equation} where $v_j=(v_{jl})_{1\leq l\leq d_j}=\widehat{v}(j,l),$  and $j\in\mathbb{N},$ and the same for $w.$ Moreover, suppose that for all $v\in H^{\infty}_{E},$ \eqref{tensoreq} holds. Then we must have $w\in\widehat{H^{\infty}_E}.$ Also, if for all $w\in\widehat{H^{\infty}_E},$ \eqref{tensoreq} holds, we have $v\in H^{\infty}_E.$
\end{lemma}

\begin{proof}
Let us  assume first that $$\sum_{j=1}^{\infty}||v_j||_{\mathtt{HS}}||w_j||_{\mathtt{HS}}<\infty,$$ for  $v\in H^{\infty}_E$ or $w\in \widehat{H^{\infty}_E}$.
 We observe that
 $$\sum_{l=1}^{d_j}|v_{jl}||w_{jl}|\leq ||v_{j}||_{\mathtt{HS}}||w_j||_{\mathtt{HS}}.$$
 And so we have
 $$\sum_{j=1}^{\infty}\sum_{l=1}^{d_j}|{v}_{jl}||w_{jl}|\leq \sum_{j=1}^{\infty}||v_j||_{\mathtt{HS}}||w_j||_{\mathtt{HS}}<\infty.$$
 
 Now we prove the other direction. 
  Here we will use the following inequality,
 $$\sum_{i=1}^n |a_i|\sum_{l=1}^n| b_{l}|=\sum_{i=1}^n |a_i| |b_{i}|+\sum_{i=1}^n |a_{i}|\left(\sum_{l=1}^n |b_{l}|-|b_{i}|\right),$$ for any $a_i,b_i\in\mathbb{R},$ yielding 
 \begin{eqnarray}||v_j||_{\mathtt{HS}}||w_j||_{\mathtt{HS}}&\leq&\sum_{i=1}^{d_j}|w_{ji}|\sum_{l=1}^{d_j}|v_{jl}|\nonumber\\
 &=&\sum_{i=1}^{d_j}|v_{ji}||w_{ji}|+\sum_{i=1}^{d_{j}}|w_{ji}|\left(\sum_{l=1}^{d_{j}}|v_{jl}|-|v_{ji}|\right).\end{eqnarray}
 We consider the second term of the above inequality, that is,
 $$\sum_{i=1}^{d_{j}}|w_{ji}|\left(\sum_{l=1}^{d_{j}}|v_{jl}|-|v_{ji}|\right)\leq \sum_{i=1}^{d_{j}}|w_{ji}|\left(Cd_j\lambda_{j}^{-s}\right),~~\forall s\in\mathbb{R},$$ since $v\in  H^{\infty}_E$ and so $v\in H^{s}_{E}, \forall s\in\mathbb{R}.$  Then  we get
 \begin{eqnarray}\sum_{j=1}^\infty||v_j||_{\mathtt{HS}}||w_j||_{\mathtt{HS}}&\leq& \sum_{j=1}^{\infty}\sum_{i=1}^{d_{j}} |w_{ji}|\left(|v_{ji}|+Cd_j\lambda_{j}^{-s}\right).
\end{eqnarray}
Let $|u_{ji}|=|v_{ji}|+Cd_j\lambda_{j}^{-s},$ for $i=1,2,...,d_j.$ Then 
\begin{eqnarray}\label{defnu}
\sum\limits_{i=1}^{d_j}|u_{ji}|^2&=&\sum\limits_{i=1}^{d_j}|v_{ji}|^2+\sum\limits_{i=1}^{d_j}C^2d_j^2\lambda_{j}^{-2s}+\sum\limits_{i=1}^{d_j}2Cd_j\lambda_{j}^{-s}|v_{ji}|\nonumber\\
&=&||v_j||_{\mathtt{HS}}^2+ C^2d_j^3\lambda_{j}^{-2s}+2Cd_j^2\lambda_{j}^{-2s}.
\end{eqnarray}
Then for any $t>0$ we have, using \eqref{defnu}, that
\begin{eqnarray}
\sum_{j=1}^{\infty}\lambda_{j}^{2t}||u_j||^{2}_{\mathtt{HS}}&=&\sum_{j=1}^{\infty}\lambda_{j}^{2t}||v_j||^{2}_{\mathtt{HS}}+C^2\sum_{j=1}^{\infty}\lambda_{j}^{2t}d_j^3\lambda_{j}^{-2s}+2C\sum_{j=1}^{\infty}\lambda_{j}^{2t}d_j^2\lambda_{j}^{-2s}\nonumber\\
&\leq& \sum_{j=1}^{\infty}\lambda_{j}^{2t}||v_j||^{2}_{\mathtt{HS}}+C^2\sum_{j=1}^{\infty}d_j\lambda_{j}^{2t+2s_0}\lambda_{j}^{-2s}\nonumber\\
&&\quad+2C\sum_{j=1}^{\infty}\lambda_{j}^{2t+2s_0}\lambda_{j}^{-2s}.
\end{eqnarray}

Now since $v\in H^{\infty}_E,$ in particular we can have $2s=2t+3s_0,$ for any $t>0,$ which gives, using \eqref{dimineq}
$$\sum_{j=1}^{\infty}\lambda_{j}^{2t}||u_j||^{2}_{\mathtt{HS}}\leq \sum_{j=1}^{\infty}\lambda_{j}^{2t}||v_j||^{2}_{\mathtt{HS}}+C^2 \sum_{j=1}^{\infty}d_j\lambda_{j}^{-s_0}+2C\sum_{j=1}^{\infty}\lambda_{j}^{-s_0}<\infty.$$
So $u\in H^{\infty}_E.$ Then we have for $w\in \widehat{H^{\infty}_E},$
$$\sum_{j=1}^\infty||v_j||_{\mathtt{HS}}||w_j||_{\mathtt{HS}}\leq \sum_{j=1}^{\infty}\sum_{i=1}^{d_{j}} |w_{ji}||u_{ji}|<\infty,$$ completing the proof.
\end{proof}

We next prove the adjointness theorem, also recalling Definition \ref{seqdefi}.
Let $\mathcal{H}, \mathcal{G}$ be two Hilbert spaces and $E$ and $F$ be the operators  defined by \eqref{ellidef} corresponding to the bases $\{e_{j}\}_{j\in\mathbb{N}},$ $\{h_{k}\}_{k\in\mathbb{N}},$ respectively, where  $d_{j}=\dim X_j$ and $g_k=\dim Y_k,$ and $X_j=\text{span}\{e_{jl}\}_{l=1}^{d_j},$ $Y_{k}=\text{span}\{h_{ki}\}_{i=1}^{g_k}.$ Also $\mathcal{H}=\bigoplus\limits_{j\in\mathbb{N}}X_j$ and $\mathcal{G}=\bigoplus\limits_{k\in\mathbb{N}}Y_k.$  We denote the corresponding  spaces to the operators $E$ and $F$ in the Hilbert space $\mathcal{H}$  and $\mathcal{G}$ respectively by $H^{\infty}_E$ and $G^{\infty}_{F}.$
\begin{theorem}{\label{THM:adj}}
A linear mapping $f:H^{\infty}_E\rightarrow G^{\infty}_{F}$ is sequential  if and only if  $f$ is represented by an infinite tensor $(f_{kjli}), $ ~ $k,j\in \mathbb{N},$ $1\leq l\leq d_{j}$ and $1\leq i\leq g_k$ such that for any $u\in H^{\infty}_E$ and $v\in\widehat{G^{\infty}_{F}}$ we have
\begin{equation} \label{EQ:f1}
\sum_{j=1}^{\infty}\sum_{l=1}^{d_j}|f_{kjli}||\widehat{u}(j,l)|<\infty, ~~\textrm{for ~all}~k\in\mathbb{N}, ~i=1,2,...,g_k, 
\end{equation} and 
\begin{equation}\label{EQ:f2}
\sum_{k=1}^{\infty}\sum_{i=1}^{g_k}\left|(v_k)_i\right|\left|{\left(\sum_{j=1}^{\infty}f_{kj}\widehat{u}(j)\right)_{i}}\right|<\infty.
\end{equation}
Furthermore, the adjoint mapping $\widehat{f}:\widehat{G^{\infty}_F}\rightarrow \widehat{H^{\infty}_E}$ defined by the formula $\widehat{f}(v)=v\circ f$ is also sequential, and the transposed  matrix ${(f_{kj})}^{t}$ represents $\widehat{f}$, with $f$ and $\widehat f$ related by $\langle f(u),v\rangle_{G^{\infty}_{F}}=\langle u,\widehat f (v)\rangle_{H^{\infty}_{E}}.$
\end{theorem}

Let us summarise the ranges for indices in the used notation as well as give more explanation to \eqref{EQ:f2}. 
For $f: H^{\infty}_E\rightarrow G^{\infty}_F$ and $u\in H^{\infty}_E$ we write
\begin{equation}\label{EQ:notf}
\mathbb C^{g_k}\ni f(u)_k=\sum_{j=1}^{\infty}f_{kj}\widehat{u}(j)=
\sum_{j=1}^\infty \sum_{l=1}^{d_j} f_{kjl}\widehat{u}(j,l),\quad k\in\mathbb{N},
\end{equation} 
so that
\begin{equation}\label{EQ:notf2}
f_{kjl}\in \mathbb C^{g_k},\;
f_{kjli}\in\mathbb C,\quad k,j\in \N ,\; 1\leq l\leq d_j,\; 1\leq i \leq g_k,
\end{equation} 
and 
\begin{equation}\label{EQ:notf3}
\mathbb C\ni (f(u)_k)_i=f(u)_{ki} = \sum_{j=1}^\infty\sum_{l=1}^{d_j} f_{kjli}\widehat{u}(j,l),\quad k\in\N,\; 1\leq i \leq g_k,
\end{equation} 
where we view $f_{kj}$ as a matrix, $f_{kj}\in\mathbb{C}^{g_k\times d_j}$, and the product of the matrices has been explained in \eqref{EQ:notf}.

\begin{rem}
Let us now describe how the tensor $(f_{kjli})$, $k,j\in \mathbb{N},$ $1\leq l\leq d_j$,  $1\leq i\leq g_k$, is constructed given a sequential mapping $f: H^{\infty}_E\rightarrow G^{\infty}_F$.
For every $k\in \mathbb{N}$ and $1\leq i\leq g_k$, define the family
$v^{ki}=\left(v^{ki}_{j}\right)_{j\in\mathbb{N}}$ such that each $v^{ki}_{j}\in \C^{d_j}$ is defined by
\begin{equation}\label{EQ:defv}
 v^{ki}_{j}(l)=\begin{cases}
               1,~~~~~j=k, l=i,\\
               0,~~~~~ \textrm{otherwise}.
            \end{cases}
\end{equation} 
 Then  $v^{ki}\in \widehat{G^{\infty}_F}$, and since
$f$ is sequential we have  $v^{ki}\circ f\in\widehat{H^{\infty}_E}$, and we can write $v^{ki}\circ f=\left(v^{ki}\circ f\right)_{j\in\N},$ where $(v^{ki}\circ f)_{j}\in\C^{d_j}.$  
Then for each $1\leq l \leq d_j$ we set
\begin{equation}\label{EQ:deff}
f_{kjli}:=(v^{ki}\circ f)_{j}(l),
\end{equation} 
the $l^{th}$ component of the vector $(v^{ki}\circ f)_{j}\in\C^{d_j}.$
The formula \eqref{EQ:deff} will be shown in the proof of Theorem {\ref{THM:adj}}.
In particular, since for $\phi\in H^{\infty}_{E}$ we have $f(\phi)\in G^{\infty}_F,$
it will be a consequence of \eqref{EQ: 4.27} and \eqref{EQ: 4.28} later on that
\begin{equation}
\label{EQ:deff2}
v^{ki}\circ f(\phi)=(\widehat{f(\phi)})(k,i)=\sum_{j=1}^{\infty}\sum_{l=1}^{d_j}f_{kjli}\widehat{\phi}(j,l),
\end{equation}
so that the tensor $(f_{kjli})$ is describing the transformation of the Fourier coefficients of $\phi$ into those of $f(\phi)$.
\end{rem}

To prove Theorem {\ref{THM:adj}} we first establish the following lemma.

\begin{lemma} \label{L:L1} 
Let  $f:H^{\infty}_E\rightarrow G^{\infty}_{F}$ be a linear mapping represented by an infinite tensor $(f_{kjli})_{k,j\in{\N}, 1\leq l\leq d_j, 1\leq i\leq g_k}$ satisfying \eqref{EQ:f1} and  \eqref{EQ:f2}. Then for all $u\in H^{\infty}_E$ and $v\in \widehat{G^{\infty}_F},$ we have
$$\lim_{n\rightarrow\infty}\sum_{k=1}^{\infty}\sum_{i=1}^{g_k}\left|(v_k)_i\right|\left|{\left(\sum_{1\leq j\leq n}f_{kj}\widehat{u}(j)\right)_{i}}\right|=\sum_{k=1}^{\infty}\sum_{i=1}^{g_k}\left|(v_k)_i\right|\left|{\left(\sum_{j=1}^{\infty}f_{kj}\widehat{u}(j)\right)_{i}}\right|.$$
\end{lemma}

\begin{proof}[Proof of Lemma \mbox{\ref{L:L1}}]

 Let $u\in H^{\infty}_E$ and $u\approx \left(\widehat u(j)\right)_{j\in {\mathbb{N}}}.$ Define $u^{n}:=\left(\widehat u ^{(n)}(j)\right)_{j\in {\mathbb{N}}}$  by setting  
\begin{equation}
\widehat{u}^{(n)}(j)=\begin{cases}
\widehat u(j), \; j\leq n, \nonumber\\
 0, \quad\;\, j>n.\end{cases}\nonumber
\end{equation} 
 Then for any  $w\in \widehat{H^{\infty}_E},$ we get  $\langle u-u^{n}, w\rangle_{H^{\infty}_E}\rightarrow 0$ as $n\rightarrow \infty.$ This is true since $\sum_{j=1}^{\infty}\left|\widehat{u}(j)\cdot w_j \right|<\infty$ so that  
 $$\left| \langle u-u^{n}, w\rangle_{H^{\infty}_E}\right|\leq \sum_{j\geq n}\left|\widehat{u_j}\cdot w_{j}\right|\rightarrow {0}$$ as $n\rightarrow \infty.$ 
  Now for any $u\in H^{\infty}_E$  and $v\in \widehat{G^{\infty}_F} $ and from \eqref{EQ:f1} and \eqref{EQ:f2} we have 
   \begin{multline}\label{EQ:long}
  \langle f(u), v \rangle_{G^{\infty}_{F}}=\sum_{k=1}^{\infty} \left(f(u)\right)_{k}\cdot v_{k}
=\sum_{k=1}^{\infty}\left(\sum_{j=1}^{\infty}f_{kj}\widehat{u}(j)\right) \cdot v_{k}
\\
=\sum_{k=1}^{\infty}\sum_{j=1}^\infty\sum_{\ell=1}^{d_j}\sum_{i=1}^{g_k} f_{kj\ell i}\widehat{u}(j,\ell)(v_k)_{i}
=\sum_{j=1}^{\infty}\sum_{\ell=1}^{d_j}\widehat{u}(j,\ell) \sum_{k=1}^{\infty}\sum_{i=1}^{g_k}
f_{kj\ell i}(v_k)_i
\\
=\sum_{j=1}^{\infty}\sum_{\ell=1}^{d_j}\widehat{u}(j,\ell) \sum_{k=1}^{\infty}
f_{kj\ell }\cdot v_k
=\sum_{j=1}^\infty\widehat{u}(j)\cdot (v\circ f)_{j}=\langle u, v\circ f\rangle_{H^{\infty}_E},
  \end{multline}
  where  
 $$\mathbb C^{d_j}\ni (v\circ f)_{j}=\left\{\sum_{k=1}^{\infty}
f_{kj\ell }\cdot v_k\right\}_{\ell=1}^{d_j},\quad j\in \mathbb{N},$$ and 
$$v\circ f=\left\{(v\circ f)_{j}\right\}_{j=1}^{\infty}.$$

Now we have the mapping $f: H^{\infty}_E\rightarrow G^{\infty}_F$ and $v\in \widehat{G^{\infty}_F},$ so we have $v\circ f: {H^{\infty}_E} \rightarrow \Sigma,$ where $\Sigma=\left\{v=(v_j)_{j\in\mathbb{N}}, v_j\in \mathbb{C}^{d_j}\right\}.$ For any 
$u\in H^{\infty}_E$ and $v\in \widehat{G^{\infty}_F}$ we have
$$\langle f(u), v\rangle_{G^{\infty}_{F}}=\sum_{j=1}^{\infty}\sum_{\ell=1}^{d_j}\sum_{k=1}^{\infty}\sum_{i=1}^{g_k}
f_{kj\ell i}\widehat{u}(j,\ell) (v_k)_i=\langle u, v\circ f\rangle_{H^{\infty}_E}.$$ 
Then from \eqref{EQ:f1} and \eqref{EQ:f2}  we have $\left\langle u,(v\circ f)\right\rangle<\infty,$ and since $H^{\infty}_E$ is perfect, we have for any $u\approx(\widehat{u}(j))_{j\in\N}\in H^{\infty}_E$  that the series  $\sum\limits_{j=1}^{\infty}\left|(v\circ f)_{j}\cdot\widehat{u}(j)\right|$ is convergent. So then  $v\circ f\in \widehat{H^{\infty}_E}.$  
Then we have
$$\langle f(u)-f(u^{n}), v\rangle_{G^{\infty}_{F}}=\langle u-u^{n}, v\circ f\rangle_{H^{\infty}_E}\rightarrow 0$$ as $n\rightarrow \infty.$ Therefore,
$$\langle f(u), v\rangle_{G^{\infty}_{F}}=\lim_{n\rightarrow\infty}\langle f(u^{n}), v\rangle_{G^{\infty}_{F}},$$ for all $u\in H^{\infty}_E$ and $v\in\widehat{G^{\infty}_F}.$
Hence for any $u\in H^{\infty}_E$ and $v\in \widehat{G^{\infty}_F}$ we have
$$\lim_{n\rightarrow\infty}\sum_{k=1}^{\infty}v_k\cdot\left(\sum_{1\leq j\leq n}f_{kj}\widehat{u}(j)\right) =\sum_{k=1}^{\infty}v_k\cdot\left(\sum_{j=1}^{\infty}f_{kj}\widehat{u}(j)\right),$$
that is,
$$\lim_{n\rightarrow\infty}\sum_{k=1}^{\infty}\sum_{i=1}^{g_k}(v_k)_i\left(\sum_{1\leq j\leq n}f_{kj}\widehat{u}(j)\right)_i=\sum_{k=1}^{\infty}\sum_{i=1}^{g_k}(v_k)_{i}\left(\sum_{j=1}^{\infty}f_{kj}\widehat{u}(j)\right)_i.$$ Now we will use the fact that if $u\in H^{\infty}_E$ then $|u|\in H^{\infty}_E$ where $|u|=\left(\widehat{ |u}|_j\right)_{j\in\mathbb{N}},$ $\widehat{ |u|}_{j}\in \mathbb{R}^{d_j},$  with
\begin{align}
    \widehat {|u|}_{j} &:= \begin{bmatrix}
          |\widehat{ u}(j,1)| \\
          |\widehat{ u}(j,2)|  \\
           \vdots \\
          |\widehat {u}(j,d_j)|
         \end{bmatrix},\nonumber
  \end{align} in view of  Theorem \ref{P:perfect}. The 
 same is true for the dual space $\left[G^{\infty}_F\right]^{\wedge}.$ 
 So then this argument gives
$$\lim_{n\rightarrow\infty}\sum_{k=1}^{\infty}\sum_{i=1}^{g_k}\left|(v_k)_i\right|\left|{\left(\sum_{1\leq j\leq n}f_{kj}\widehat{u}(j)\right)_{i}}\right|=\sum_{k=1}^{\infty}\sum_{i=1}^{g_k}\left|(v_k)_i\right|\left|{\left(\sum_{j=1}^{\infty}f_{kj}\widehat{u}(j)\right)_{i}}\right|.$$ The proof is complete.
\end{proof}
\begin{remark} This proof does not require sequentiality and it can be used to improve the argument  in \cite[Theorem 4.7]{DaR3}.
\end{remark}

\begin{proof} [Proof of Theorem \mbox{\ref{THM:adj}}]  
 Let us assume first that the mapping $f: H^{\infty}_E\rightarrow G^{\infty}_F$ can be represented by $f=(f_{kjli})_{k,j\in\mathbb{N},
 1\leq l\leq d_j, 1\leq i\leq g_k},$ an infinite tensor such that\begin{equation}\sum_{j=1}^{\infty}\sum_{l=1}^{d_j}|f_{kjli}||\widehat{u}(j,l)|<\infty, ~~\textrm{for ~all}~k\in\mathbb{N},~ i=1,2,\ldots,g_k,\end{equation}  and
 \begin{equation}\sum_{k=1}^{\infty}\sum_{i=1}^{g_k}\left|(v_k)_i\right|\left|{\left(\sum_{j=1}^{\infty}f_{kj}\widehat{u}(j)\right)_{i}}\right|<\infty\end{equation} hold for all $u\in H^{\infty}_E$ and $v\in\widehat{G^{\infty}_F}.$
 
 Let $\widehat{u}_{1}=\left(\widehat{u_1}(p)\right)_{p\in{\mathbb{N}}}$ be such that for some $j,l$ where $j\in{\mathbb{N}},$ $1\leq l\leq d_j$,  we have 
 \begin{equation} \widehat{u_1}(p,q)=\begin{cases}
 {1},~~~~ p=j, \; q=l,\\
 0, ~~~~~~\textrm{otherwise}.\end{cases}\nonumber\end{equation}
 
 Then $u_{1}\in H^{\infty}_E$  so $fu_1=f(u_{1})\in G^{\infty}_F$ and
 \begin{eqnarray}\label{EQ:4.21}
 \left(fu_1\right)_{k}&=&\sum_{p=1}^{\infty}f_{kp}\widehat{u}_{1}(p)\nonumber\\
 &=&\sum_{p=1}^{\infty}\sum_{q=1}^{d_p}f_{kpq}\widehat{u_{1}}(p,q)\nonumber\\
 &=& \sum_{q=1}^{d_j}f_{kjq}\widehat{u_1}(j,q)\nonumber\\
 &=&f_{kjl}\in \C^{g_k}.
 \end{eqnarray}

 We now first show that
 $$\widehat{\left(fu\right)}(k)=\sum_{j=1}^{\infty}\sum_{l=1}^{d_j}f_{kjl}\widehat{u}(j,l),$$ where $f_{kjli}\in \mathbb{C}$ for each $k,j\in\mathbb{N},$ $1\leq l \leq d_j$ and $1\leq i\leq g_k.$
The way in which $f$ has been defined we have
 $$(fu)_{k}= \sum_{j=1}^{\infty}\sum_{l=1}^{d_j}f_{kjl}\widehat{u}(j,l),\quad f_{kjl}\in \C^{g_k}.$$
Also since $u\in H^{\infty}_E$, from our assumption we have $fu\in G^{\infty}_F$  and $fu\approx \left(\widehat{(fu)}(j)\right)_{j\in\mathbb{N}}$, so that $(fu)_{k}\approx\widehat{(fu)}(k).$

We  can then write $\widehat{(fu)}(k)=\sum\limits_{j=1}^{\infty}\sum\limits_{l=1}^{d_j}f_{kjl}\widehat{u}(j,l).$
Since we know that $v\in{\left[G^{\infty}_F\right]^{\wedge}}$ and $fu\in G^{\infty}_F,$ we have
$$\sum_{k=1}^{\infty}\sum_{i=1}^{g_k}|(v_k)_i||(\widehat{(fu)}(k))_i|=\sum_{k=1}^{\infty}\sum_{i=1}^{g_k}|(v_k)_i||\sum_{j=1}^{\infty}\sum_{l=1}^{d_j} f_{kjli}\widehat{u}(j,l)|<\infty.$$
 In particular  using the definition of $u_1$ and \eqref{EQ:4.21} we get
\begin{eqnarray}\label{EQ:4.22}\sum_{k=1}^{\infty}\sum_{i=1}^{g_k}|(v_k)_i|\left|\sum_{p=1}^{\infty} \sum_{q=1}^{d_p}f_{kpqi}\widehat{u_1}(p,q)\right|=\sum_{k=1}^{\infty}\sum_{i=1}^{g_k}|(v_k)_i||f_{kjli}|<\infty,\end{eqnarray}
for any $j\in \mathbb{N}$ and $1\leq l\leq d_j.$

Now for any $u\in H^{\infty}_E$ consider 
$$J=\sum_{j=1}^{\infty}\sum_{l=1}^{d_j}|\sum_{k=1}^{\infty}\sum_{i=1}^{g_k}(v_{k})_{i}f_{kjli}| |\widehat{u}(j,l)|.$$
Then we consider the series 
$$I_{n}:=\sum_{1\leq j\leq n}\sum_{l=1}^{d_j}|\sum_{k=1}^{\infty}\sum_{i=1}^{g_k}(v_{k})_{i}f_{kjli}| |\widehat{u}(j,l)|,$$
so that we have
\begin{eqnarray}
I_{n}&=&\sum_{1\leq j\leq n}\sum_{l=1}^{d_j}|\sum_{k=1}^{\infty}\sum_{i=1}^{g_k}(v_{k})_{i}f_{kjli}| |\widehat{u}(j,l)|\nonumber\\
&=&\sum_{1\leq j\leq n}\sum_{l=1}^{d_j}|\sum_{k=1}^{\infty}\sum_{i=1}^{g_k}(v_{k})_{i}f_{kjli}\widehat{u}(j,l)|.\nonumber
\end{eqnarray}
Let $\epsilon=(\epsilon_i)_{1\leq i\leq d_{k}},$ $k\in{\mathbb{N}}$, be  such that $\epsilon_i\in\mathbb{C}$ and $|\epsilon_i|\leq C, $ for  all $i$ and  such that
$$|\sum_{k=1}^{\infty}\sum_{i=1}^{g_k}(v_{k})_{i}f_{kjli}\widehat{u}(j,l)|=\sum_{k=1}^{\infty}\sum_{i=1}^{g_k}(v_{k})_{i}f_{kjli}\widehat{u}(j,l)\epsilon_{i}.$$
Then 
\begin{eqnarray}
I_n &=&\sum_{1\leq j\leq n}\sum_{l=1}^{d_j}\sum_{k=1}^{\infty}\sum_{i=1}^{g_k}(v_{k})_{i}f_{kjli}\widehat{u}(j,l)\epsilon_{i}\nonumber\\
&\leq & C\sum_{k=1}^{\infty}\sum_{i=1}^{g_k}|(v_{k})_{i}|\left|\sum_{1\leq j\leq n}\sum_{l=1}^{d_j}f_{kjli})\widehat{u}(j,l)\epsilon_{i} \right|.
\end{eqnarray}
It follows from Lemma \ref{L:L1} that
$$\lim_{n\rightarrow\infty}\sum_{k=1}^{\infty}\sum_{i=1}^{g_k}|(v_{k})_{i}|\left|\sum_{1\leq j\leq n}\sum_{l=1}^{d_j}f_{kjli}\widehat{u}(j,l)\epsilon_{i} \right| =  \sum_{k=1}^{\infty}\sum_{i=1}^{g_k}|(v_{k})_{i}|\left|\sum_{j=1}^{\infty}\sum_{l=1}^{d_j}f_{kjli}\widehat{u}(j,l)\epsilon_{i} \right|<\infty.$$

Then \begin{equation}\label{EQ:4.24} J=\sum_{j=1}^{\infty}\sum_{l=1}^{d_j}|\sum_{k=1}^{\infty}\sum_{i=1}^{g_k}(v_{k})_{i}f_{kjli}| |\widehat{u}(j,l)|<\infty.\end{equation}

So  we proved that if $(f_{kjli})$ satisfies
\begin{itemize}
\item $\sum\limits_{j=1}^{\infty}\sum\limits_{l=1}^{d_j}|f_{kjli}||\widehat{u}(j,l)|<\infty$,
\item $\sum\limits_{k=1}^{\infty}\sum\limits_{i=1}^{g_k}\left|(v_k)_i\right|\left|{\left(\sum\limits_{j=1}^{\infty}f_{kj}\widehat{u}(j)\right)_{i}}\right|<\infty$,
\end{itemize}
 then  for any $u\in{H^{\infty}_E}$ and $v\in\left[G^{\infty}_F\right]^{\wedge}$ we have from \eqref{EQ:4.22} and \eqref{EQ:4.24}, respectively,  that
 \begin{enumerate}[label=(\roman*)]
 \item $\sum\limits_{k=1}^{\infty}\sum\limits_{i=1}^{g_k}|(v_k)_i||f_{kjli}|<\infty$, 
 \item $\sum\limits_{j=1}^{\infty}\sum\limits_{l=1}^{d_j}|\sum\limits_{k=1}^{\infty}\sum\limits_{i=1}^{g_k}(v_{k})_{i}f_{kjli})| |\widehat{u}(j,l)|<\infty.$
 \end{enumerate}
Now recall that for $f: H^{\infty}_E\rightarrow G^{\infty}_F$ we have
$$(f(u))_{k}=\sum\limits_{j=1}^{\infty}\sum\limits_{l=1}^{d_j}f_{kjl}\widehat{u}(j,l),$$ 
for any $u\in{H^{\infty}_E},$ then for any  $v\in\left[G^{\infty}_F\right]^{\wedge}$, the composed mapping  $v\circ f:  H^{\infty}_E\rightarrow \mathbb{C}$ is given by
\begin{eqnarray}(v\circ f)(u)&=&\sum\limits_{k=1}^{\infty}v_k\cdot (f(u))_k=\sum\limits_{k=1}^{\infty}\sum\limits_{i=1}^{g_k}(v_{k})_{i}\left(\sum\limits_{j=1}^{\infty} \sum\limits_{l=1}^{d_j}f_{kjli}\widehat u(j,l)\right)\nonumber\\
&=&\sum\limits_{j=1}^{\infty}\sum\limits_{l=1}^{d_j}\left(\sum\limits_{k=1}^{\infty} \sum\limits_{i=1}^{g_k}(v_{k})_{i}f_{kjli}\right)\widehat u(j,l).\end{eqnarray}

So by (ii) we get that
$$\left|(v\circ f)(u)\right|\leq\sum\limits_{j=1}^{\infty}\sum_{l=1}^{d_j}|\sum\limits_{k=1}^{\infty}\sum\limits_{i=1}^{g_k}(v_{k})_{i}f_{kjli}| |\widehat{u}(j,l)|<\infty.$$
So $\widehat f(v)=(\widehat f (v)_{j l})_{j\in \mathbb N, 1\leq l\leq d_j},$ with $\widehat{f}(v)_{jl}= \sum\limits_{k=1}^{\infty}\sum\limits_{i=1}^{g_k}(v_{k})_{i}f_{kjli}\in \widehat{H^{\infty}_E}$ (from the definition of $\widehat{H^{\infty}_E}$), that is $f$ is sequential.
And then $\langle f(u), v\rangle_{G^{\infty}_F}=\langle u,\widehat f(v)\rangle_{H^{\infty}_E}$  is also true.

\medskip
Now to prove the converse part we assume that $f: H^{\infty}_E \rightarrow G^{\infty}_F$ is sequential. We have to show that $f$ can be represented as  $f\approx(f_{kjli})_{k,j\in\mathbb{N},1\leq l\leq d_j, 1\leq i\leq g_k}$ and satisfies \eqref{EQ:f1} and \eqref{EQ:f2}.

Define for  $k,i$ where $k\in \mathbb{N}$ and $1\leq i\leq g_k,$  the sequence $u^{ki}=\left(u^{ki}_{j}\right)_{j\in\mathbb{N}}$ such that $u^{ki}_{j}\in \C^{d_j}$ and $u^{ki}_j(l)=\widehat{u^{ki}}({j,l})$, given by 
\[
  u^{ki}_{j}(l)=\widehat{u^{ki}}({j,l})=\begin{cases}
               1,~~~~~j=k, l=i,\\
               0,~~~~~ \textrm{otherwise}.
            \end{cases}
\]

Then  $u^{ki}\in \left[G^{\infty}_F\right]^{\wedge}.$
Now since  $f$ is sequential we have  $u^{ki}\circ f\in\widehat{H^{\infty}_E}$ and $u^{ki}\circ f=\left(\left(u^{ki}\circ f\right)_{j}\right)_{j\in\mathbb{N}},$ where $(u^{ki}\circ f)_{j}\in\C^{d_j}.$  We denote $u^{ki}\circ f=\left(f^{ki}_{j}\right)_{j\in{\N}},$ where $f^{ki}_j=(u^{ki}\circ f)_{j}.$ Then $(f^{ki}_{j})_{j\in\N}\in \widehat{H^{\infty}_E}$ and $f^{ki}_{j}\in\C^{d_j}.$

Then for any $\phi\approx \left(\widehat{\phi}(j)\right)_{j\in\mathbb{N}}\in H^{\infty}_E$ we have
\begin{equation}\sum_{j=1}^{\infty}\sum_{l=1}^{d_j}|f^{ki}_{jl}||\widehat{\phi}(j,l)|<\infty.\end{equation}
 
For $\phi\in H^{\infty}_E$ we can write $f(\phi)\in G^{\infty}_F.$ We can also write $$f(\phi)=\left(\widehat{f(\phi)}(p)\right)_{p\in\N}.$$ So
\begin{eqnarray}\label{EQ: 4.27}
u^{ki}\circ f(\phi)&=&\sum_{j=1}^{\infty}u^{ki}_j\widehat{(f(\phi))}_j\nonumber\\
&=&\sum_{j=1}^{\infty}\sum_{l=1}^{d_j}u^{ki}_{jl}\widehat{(f(\phi))}(j,l)\nonumber\\
&=&(\widehat{f(\phi)})(k,i)~(\textrm{from~ the ~definition~ of~} u^{ki}).\end{eqnarray}

We have  $u^{ki}\circ f=(f^{ki})\in \widehat{H^{\infty}_E},$ so 
\begin{eqnarray}\label{EQ: 4.28}
(u^{ki}\circ f)(\phi)
&=&\sum_{j=1}^{\infty}f^{ki}_{j}\widehat{\phi}(j)\nonumber\\
&=&\sum_{j=1}^{\infty}\sum_{l=1}^{d_j}f^{ki}_{jl}\widehat{\phi}(j,l).
\end{eqnarray}
From \eqref{EQ: 4.27} and \eqref{EQ: 4.28} we have $(\widehat{f(\phi)})(k,i)=\sum\limits_{j=1}^{\infty}\sum\limits_{l=1}^{d_j}f^{ki}_{jl}\widehat{\phi}(j,l).$\\
Hence  $(f(\phi))_{ki}=\sum\limits_{j=1}^{\infty}\sum\limits_{l=1}^{d_j}f^{ki}_{jl}\widehat{\phi}(j,l),~~k\in\mathbb{N},$ and $1\leq i\leq g_k,$ that is  $f$ is represented by the tensor $\left\{(f^{ki}_{jl})\right\}_{k,j\in\mathbb{N}, 1\leq i\leq g_k, 1\leq l\leq d_j}$.\\
If we denote $f^{ki}_{jl}$ by $f^{ki}_{jl}= f_{kjli},$ we can say that $f$ is represented by the tensor $(f_{kjli})_{k,j\in\mathbb{N}, 1\leq l\leq d_j, 1\leq i\leq g_k}.$ 
Also let $v\in\widehat{G^{\infty}_F}.$ Since $f(\phi)\in G^{\infty}_F$ for $\phi\in H^{\infty}_E,$ then from the definition of $\widehat{G^{\infty}_F}$  we have
$$\sum_{k=1}^{\infty}\sum_{i=1}^{g_k}|(v_k)_i|\sum_{j=1}^{\infty}\sum_{l=1}^{d_j}f_{kjli}\widehat{\phi}(j,l)|<\infty.$$
This completes the proof of Theorem \ref{THM:adj}.
\end{proof}

\section{Applications to universality}\label{Universality}

In this section we give an application of the developed analysis to the universality problem. We start with the spaces of smooth functions, and then make some remarks how the same arguments can be extended to the Komatsu classes setting from \cite{DaR2}.

\medskip
First we recall the notations:

\medskip
 Let $\mathcal{H}, \mathcal{G}$ be two Hilbert spaces and let $E$ and $F$ be the operators corresponding to the bases $\{e_{j}\}_{j\in\mathbb{N}},$ $\{h_{k}\}_{k\in\mathbb{N}},$ as in \eqref{ellidef}, where  $d_{j}=\dim X_j$ and $g_k=\dim Y_k,$ and $X_j=\text{span}\{e_{jl}\}_{l=1}^{d_j},$ $Y_{k}=\text{span}\{h_{ki}\}_{i=1}^{g_k}.$ Also $\mathcal{H}=\bigoplus\limits_{j\in\mathbb{N}}X_j$ and $\mathcal{G}=\bigoplus\limits_{k\in\mathbb{N}}Y_k.$\\  
  We denote the spaces of smooth type functions corresponding to the operators $E$ and $F$ in the Hilbert space $\mathcal{H}$  and $\mathcal{G},$ respectively, by $H^{\infty}_E$ and $G^{\infty}_{F}.$ \\

 The main application of  Theorem \ref{THM:adj}  will be in the setting when $X, Y$ are compact manifold without boundary, where  $\mathcal{H}=L^{2}(X)$ and $H^{\infty}_E=C^{\infty}(X)$,  and $\mathcal{G}=L^{2}(Y),$ $G^{\infty}_F=C^{\infty}(Y).$ 
 
 Using Theorem \ref{THM:adj} we prove the universality of the spaces of the smooth type functions, $C^{\infty}(X),$ where we can write
 $$f\in C^{\infty}(X)\iff \forall N\exists C_N~:~ |\widehat{f}(j,k)|\leq C_N(1+\lambda_j)^{-N}~\text{for~all~} j,k.$$ Further details of such spaces can be found in \cite{DR}. In particular, if $E$ is an elliptic pseudo-differential operator of positive order, then this is just the usual space of smooth functions on $X$.
 
\begin{defi}\label{mapdef} Let $E$ be a self-adjoint, positive operator.  A mapping $f: X \rightarrow W$ from the compact manifold $X$ to a sequence space $W$, is said to be a $H^{\infty}_E$-mapping if for any $u\in \widehat{W},$ the composed mapping $u\circ f: X\rightarrow \mathbb{C}$ belongs to $H^{\infty}_{E}$.
\end{defi}
 Next we prove the universality of the spaces of smooth type functions. 
\begin{theorem}\label{UniversalityTH} 
Let $X$ be a compact manifold.
\begin{itemize}
\item[(i)] The delta mapping $\delta: X\rightarrow \widehat{H^{\infty}_{E}}$ defined by 
$$\delta(x)=\delta_x,$$ {and} $$\delta_{x}(\phi)=\langle\delta_{x},\phi\rangle_{H^{\infty}_{E}}= \sum\limits_{j=1}^{\infty}\sum\limits_{l=1}^{d_j}\widehat{\phi}(j,l)(\delta_x)_{jl}=\phi(x), ~~~\text{for~~all}~~\phi\in H^{\infty}_{E},~x\in X,$$ is a $H^{\infty}_{E}$-mapping.
\item[(ii)] If $\tilde{g}: \widehat{H^{\infty}_{E}}\rightarrow G^{\infty}_{F}$ is a sequential linear mapping, then the composed mapping $\tilde{g}\circ\delta:  X \rightarrow G^{\infty}_{F}$ is a $H^{\infty}_{E}$-mapping.
\item[(iii)] For any $H^{\infty}_{E}$-mapping $f: X\rightarrow G^{\infty}_{F},$ there exists a unique sequential linear mapping $\widehat{f}:\widehat{H^{\infty}_{E}}\rightarrow G^{\infty}_{F}$ such that $f=\widehat{f}\circ\delta.$
\end{itemize} 
\end{theorem}
\vspace{.2cm}
\begin{proof}
(i)  ~ Recall  that $\widehat{H^{\infty}_{E}}=\bigcup\limits_{s\in\mathbb{R}}H^{-s}_{E}=[H^{\infty}_{E}]^{\prime}.$ \\
 Let  $v\in [\widehat{H^{\infty}_{E}}]^{\wedge}=H^{\infty}_{E}.$  We define  the composed mapping
$$v\circ\delta_x=\langle\delta_{x},v\rangle_{H^{\infty}_{E}}=\sum\limits_{j=1}^{\infty}\sum\limits_{l=1}^{d_j}v_{jl}(\delta_x)_{jl}=v(x) ~(\text{by~definition~of~}\delta).$$ This is well-defined since $v\in H^{\infty}_E$ and $\delta_{x}\in[H^{\infty}_{E}]^{\prime}= \widehat{H^{\infty}_{E}}.$  Also since $v\in H^{\infty}_{E}$, we see that $v\circ\delta\in H^{\infty}_{E}$ and that implies $\delta$ is a $H^{\infty}_{E}$-mapping.

\medskip
(ii)~ Let $u\in \widehat{G^{\infty}_{F}}=[G^{\infty}_F]^{\prime}.$\\  From the definition of $H^{\infty}_E$-mapping we have to show that $u\circ\tilde{g}\circ\delta\in H^{\infty}_{E}.$\\
Now by given condition $\tilde{g}: [H^{\infty}_{E}]^{\prime}\rightarrow G^{\infty}_{F},$ so we have  $u\circ\tilde{g}\in\left[ {H^{\infty}_{E}}\right]^{\prime \prime}.$ \\
Here we claim that
$\left[ {H^{\infty}_{E}}\right]^{\prime \prime}=H^{\infty}_{E}.$\\
 Note that  $H^{\infty}_{E}\subseteq \left[ {H^{\infty}_{E}}\right]^{\prime \prime}.$\\
  Recall that, $[H^{\infty}_{E}]^{\prime}=\bigcup\limits_{s\in\mathbb{R}}H^{-s}_{E}$ and so for any $s\in\mathbb{R},$ $$ H^{-s}_{E}\subseteq [H^{\infty}_{E}]^{\prime}
\Rightarrow [H^{\infty}_{E}]^{\prime\prime}\subseteq  [H^{-s}_{E}]^{\prime}
\Rightarrow [H^{\infty}_{E}]^{\prime\prime}\subseteq H^{s}_{E}.$$ Since the above is true for any $s\in\mathbb{R},$ we have $ [H^{\infty}_{E}]^{\prime\prime}\subseteq \bigcap\limits_{s\in\mathbb{R}} H^{s}_{E}=H^{\infty}_{E}$ and this gives $[H^{\infty}_{E}]^{\prime\prime}=H^{\infty}_{E}.$

Then $u\circ\tilde{g}\in H^{\infty}_{E}.$ Using same argument as in  the proof of $(i)$ and from the definition of the mapping $\delta$ we have $u\circ\tilde{g}\circ\delta_x=u\circ\tilde{g}(x)$ and $u\circ\tilde{g}\circ\delta$ belongs to $H^{\infty}_{E}.$\\ So from Definition \ref{mapdef}, $\tilde{g}\circ\delta$ is a $H^{\infty}_{E}$-mapping. 

\medskip
(iii)~\underline{Existence of  $\widehat{f}: \widehat{H^{\infty}_E}\rightarrow {G^{\infty}_F}.$}\\

By hypothesis, $f: X\rightarrow{G^{\infty}_F}$ is a  $H^{\infty}_E$-mapping so that for any $v\in  \widehat{G^{\infty}_F},$ $v\circ f\in H^{\infty}_E .$  A sequential linear mapping $\tilde{f}:\widehat{G^{\infty}_F}\rightarrow H^{\infty}_E$ can be defined by $\tilde{f}(v)=v\circ f,$ and $v(f(x))=\langle f(x), v\rangle_{G^{\infty}_F}.$\\
  Hence by Theorem \ref{THM:adj} there is an adjoint mapping, we denote it by $\widehat{f},$ where $\widehat{f}: \widehat{H^{\infty}_E}\rightarrow[\widehat{G^{\infty}_F}]^{\wedge}$ is a sequential mapping. By the definition of the adjoint mapping $\widehat{f}$ we have 
$$\langle u, \tilde{f}(v)\rangle_{H^{\infty}_E}=\langle \widehat{f}\circ u, v\rangle_{G^{\infty}_F},~~ u\in[ H^{\infty}_E]^{\wedge},$$  where $\langle\cdot,\cdot\rangle$ is the bilinear function on $H^{\infty}_E\times \widehat{H^{\infty}_E}$  defined in Section \ref{SEC:seqspaces}. The above can be written as 
$$\langle u, v\circ f\rangle_{H^{\infty}_E}=\langle \widehat{f}\circ u, v\rangle_{G^{\infty}_F},~~ u\in\widehat{H^{\infty}_E}.$$ For $u=\delta_x,$ this gives 
$$\langle \widehat{f}\circ \delta_x, v\rangle_{G^{\infty}_F}=\langle \delta_x, v\circ f\rangle_{{H^{\infty}_E}}=(v\circ f)(x)=v(f(x))=\langle f(x), v\rangle_{G^{\infty}_F},$$ for any $v\in \widehat{G^{\infty}_F}$. This proves $f=\widehat{f}\circ \delta.$\\~\\
\underline{Uniqueness of $\widehat{f}:\widehat{H^{\infty}_E}\rightarrow G^{\infty}_F .$} \\
 
 Suppose $\widehat{f}\circ\delta=f=0.$
We have to show that $\widehat{f}=0$ on $\widehat{H^{\infty}_E}.$ Since $\widehat{f}$ is sequential, there exists $g: \widehat{G^{\infty}_F}\rightarrow H^{\infty}_E$ such that
$$\langle \widehat{f}\circ u,v\rangle_{G^{\infty}_F}=\langle u, g\circ v\rangle_{{H^{\infty}_E}}.$$
Take $u=\delta_x\in \widehat{H^{\infty}_E},$ then
$$\langle \widehat{f}\circ\delta_x, v\rangle_{G^{\infty}_F} =\langle \delta_x, g\circ v\rangle_{{H^{\infty}_E}} =g(v(x))=0$$ for any $v\in \widehat{G^{\infty}_F},$ that is, $g=0$ on $\widehat{G^{\infty}_F}.$ From $\langle \widehat{f}\circ u, v\rangle_{G^{\infty}_F} =0$ for any $v\in \widehat{G^{\infty}_F},$ we get $ \widehat{f}\circ u=0$ for any $u\in \widehat{H^{\infty}_E},$ that is $\widehat{f}=0$ on $\widehat{H^{\infty}_E}.$
\end{proof}

\subsection{Extension to Komatsu classes}\hfill\\

Here we briefly outline how the analysis above can be extended to the setting of Komatsu classes from \cite{DaR2, DaR3}.

\begin{remark} 
In another work {\rm (\cite{DaR3})} we studied the Komatsu classes of ultra-differen\-tiable functions $\Gamma_{\{M_k\}}(X)$ on a compact manifold $X,$ where $M_{\{k\}}$ be a sequence of positive numbers such that
\begin{enumerate}
\item $M_0=1,$
\item $M_{k+1}\leq AH^{k}M_k, ~k=0,1,2,...,$
\item $M_{2k}\leq AH^{k}\min\limits_{0\leq q\leq k}M_qM_{k-q}, k=0,1,2,...,$ for some $A,H>0.$
\end{enumerate}
 In \cite{DaR3} we have characterised the dual spaces of these Komatsu classes and have shown that these spaces are perfect spaces, i.e, these spaces coincide with their second dual spaces. Furthermore, in \cite[Theorem 4.7]{DaR3} we proved the following theorem for Komatsu classes of functions on a compact manifold $X$:
 \end{remark}
 
\begin{theorem}[Adjointness Theorem] Let $\{M_k\}$and $\{N_k\}$ satisfy conditions $(M.0)-(M.3)$. A linear mapping $f:\Gamma_{\{M_k\}}(X)\rightarrow\Gamma_{\{N_k\}}(X)$ is sequential if and only if $f$ is represented by an infinite tensor $(f_{kjli}),$ $k,j\in\mathbb{N}_{0},$ $1\leq l\leq d_j$ and $1\leq l\leq d_k$ such that for any $u\in \Gamma_{\{M_k\}}(X)$ and $v\in \widehat{\Gamma_{\{N_k\}}(X)}$ we have
$$\sum\limits_{j=0}^{\infty}\sum\limits_{l=1}^{d_j}\left|f_{kjli}\right|\left|\widehat{u}(j,l)\right|<\infty, ~~for ~~all ~~k\in\mathbb{N}_{0},~i=1,2,...,d_k,$$
and
$$\sum\limits_{k=0}^{\infty}\sum\limits_{i=1}^{d_k}\left|(v_k)_i\right|\left|\left(\sum\limits_{j=0}^{\infty}f_{kj}\hat{u}(j)\right)_i\right|<\infty.$$
Furthermore, the adjoint mapping $\widehat{f}: \widehat{\Gamma_{\{N_k\}}(X)}\rightarrow \widehat{\Gamma_{\{M_k\}}(X)}$ defined by the formula $\widehat{f}(v)=v\circ f$ is sequential, and the transposed matrix $(f_{kj})^{t}$ represents $\widehat{f},$ with $f$ and $\widehat{f}$ related by $\langle f(u), v\rangle=\langle u,\widehat{f}(v)\rangle.$
\end{theorem} 

The above theorem described the tensor structure of sequential mappings on spaces of Fourier coefficients and characterised their adjoint mappings. 
Now in particular the considered classes include spaces  of analytic and Gevrey functions (which are perfect spaces too), as well as spaces of ultradistributions, yielding tensor representations for linear mappings between these spaces on compact manifolds. Now using \cite[Theorem 4.7]{DaR3} and the same techniques used in this paper to prove the  universality of smooth functions in Theorem \ref{UniversalityTH}, on compact manifolds in Section \ref{Universality}, one also obtains the universality of the Gevrey classes of ultradifferentiable functions on compact groups (from \cite{DaR1}) and Komatsu classes of functions in compact manifolds. As the proof would be a repetition of those arguments, we omit it here.

\end{document}